\documentclass{amsart}
\usepackage{amsfonts}

\newtheorem{theorem}{Theorem}
\theoremstyle{plain}

\newtheorem{corollary}{Corollary}

\newtheorem{definition}{Definition}

\newtheorem{lemma}{Lemma}

\newtheorem{remark}{Remark}

\numberwithin{equation}{section}
\input{tcilatex}

\begin{document}
\title[Ostrowski type inequalities]{Ostrowski type inequalities for
harmonically $s$-convex functions via fractional integrals}
\author{\.{I}mdat \.{I}\c{s}can}
\address{Department of Mathematics, Faculty of Arts and Sciences,\\
Giresun University, 28100, Giresun, Turkey.}
\email{imdati@yahoo.com, imdat.iscan@giresun.edu.tr}

\begin{abstract}
In this paper, a new identity for fractional integrals is established. Then
by making use of the established identity, some new Ostrowski type
inequalities for harmonically $s$-convex functions via Riemann--Liouville
fractional integral are established.
\end{abstract}

\maketitle

\section{Introduction}

Let $f:I\mathbb{\rightarrow R}$, where $I\subseteq \mathbb{R}$ is an
interval, be a mapping differentiable in $I^{\circ }$ (the interior of $I$)
and let $a,b\in I^{\circ }$ with $a<b.$ If $\left\vert f^{\prime
}(x)\right\vert \leq M,$ for all $x\in \left[ a,b\right] ,$ then the
following inequality holds%
\begin{equation}
\left\vert f(x)-\frac{1}{b-a}\int_{a}^{b}f(t)dt\right\vert \leq M(b-a)\left[ 
\frac{1}{4}+\frac{\left( x-\frac{a+b}{2}\right) ^{2}}{\left( b-a\right) ^{2}}%
\right]  \label{1-1}
\end{equation}%
for all $x\in \left[ a,b\right] .$ This inequality is known in the
literature as the Ostrowski inequality (see \cite{O38}), which gives an
upper bound for the approximation of the integral average $\frac{1}{b-a}%
\int_{a}^{b}f(t)dt$ by the value $f(x)$ at point\ $x\in \left[ a,b\right] $.
For some results which generalize, improve and extend the inequalities(\ref%
{1-1}) we refer the reader to the recent papers (see \cite{ADDC10,I13,L12} ).

In \cite{HM94}, Hudzik and Maligranda considered the following class of
functions:

\begin{definition}
A function $f:I\subseteq 
\mathbb{R}
_{+}\rightarrow 
\mathbb{R}
$ where $%
\mathbb{R}
_{+}=\left[ 0,\infty \right) $, is said to be $s$-convex in the second sense
if%
\begin{equation*}
f\left( \alpha x+\beta y\right) \leq \alpha ^{s}f(x)+\beta ^{s}f(y)
\end{equation*}%
for all $x,y\in I$ and $\alpha ,\beta \geq 0$ with $\alpha +\beta =1$ and $s$
fixed in $\left( 0,1\right] $. They denoted this by $K_{s}^{2}.$
\end{definition}

It can be easily seen that for $s=1$, $s$-convexity reduces to ordinary
convexity of functions defined on $[0,\infty )$. For some recent results and
generalizations concerning $s$-convex functions see \cite%
{ADK11,AKO11,DF99,HBI09,I13,I13b,KBO07,L12,S12}.

In \cite{I13a}, the author gave harmonically convex and established
Hermite-Hadamard's inequality for harmonically convex functions as follows:

\begin{definition}
Let $I\subset 
\mathbb{R}
\backslash \left\{ 0\right\} $ be a real interval. A function $%
f:I\rightarrow 
\mathbb{R}
$ is said to be harmonically convex, if \ 
\begin{equation}
f\left( \frac{xy}{tx+(1-t)y}\right) \leq tf(y)+(1-t)f(x)  \label{1-3}
\end{equation}%
for all $x,y\in I$ and $t\in \lbrack 0,1]$. If the inequality in (\ref{1-3})
is reversed, then $f$ is said to be harmonically concave.
\end{definition}

\begin{theorem}
Let $f:I\subset 
\mathbb{R}
\backslash \left\{ 0\right\} \rightarrow 
\mathbb{R}
$ be a harmonically convex function and $a,b\in I$ with $a<b.$ If $f\in
L[a,b]$ then the following inequalities hold 
\begin{equation}
f\left( \frac{2ab}{a+b}\right) \leq \frac{ab}{b-a}\dint\limits_{a}^{b}\frac{%
f(x)}{x^{2}}dx\leq \frac{f(a)+f(b)}{2}.  \label{1-4}
\end{equation}%
The \ above inequalities are sharp.
\end{theorem}

Let us recall the following special functions:

(1) The Beta function:%
\begin{equation*}
\beta \left( x,y\right) =\frac{\Gamma (x)\Gamma (y)}{\Gamma (x+y)}%
=\dint\limits_{0}^{1}t^{x-1}\left( 1-t\right) ^{y-1}dt,\ \ x,y>0,
\end{equation*}

(3) The hypergeometric function:%
\begin{equation*}
_{2}F_{1}\left( a,b;c;z\right) =\frac{1}{\beta \left( b,c-b\right) }%
\dint\limits_{0}^{1}t^{b-1}\left( 1-t\right) ^{c-b-1}\left( 1-zt\right)
^{-a}dt,\ c>b>0,\ \left\vert z\right\vert <1\text{ (see \cite{KST06}).}
\end{equation*}

\begin{definition}[\protect\cite{I13e}]
Let $I\subset \left( 0,\infty \right) $ be an real interval. A function $%
f:I\rightarrow 
\mathbb{R}
$ is said to be harmonically $s$-convex (concave), if \ 
\begin{equation*}
f\left( \frac{xy}{tx+(1-t)y}\right) \leq \left( \geq \right)
t^{s}f(y)+(1-t)^{s}f(x)
\end{equation*}%
for all $x,y\in I$ , $t\in \lbrack 0,1]$ and for some fixed $s\in \left( 0,1%
\right] $.
\end{definition}

The following identity is proved by Iscan (see \cite{I13e}).

\begin{lemma}
\label{1.5}Let $f:I\subset 
\mathbb{R}
\backslash \left\{ 0\right\} \rightarrow 
\mathbb{R}
$ be a differentiable function on $I^{\circ }$ and $a,b\in I$ with $a<b$. If 
$f^{\prime }\in L[a,b]$ then for all $x\in \left[ a,b\right] $ we have 
\begin{eqnarray}
&&f(x)-\frac{ab}{b-a}\dint\limits_{a}^{b}\frac{f(u)}{u^{2}}du  \label{1-5} \\
&=&\frac{ab}{b-a}\left\{ \left( x-a\right) ^{2}\dint\limits_{0}^{1}\frac{t}{%
\left( ta+(1-t)x\right) ^{2}}f^{\prime }\left( \frac{ax}{ta+(1-t)x}\right)
dt\right.   \notag \\
&&-\left. \left( b-x\right) ^{2}\dint\limits_{0}^{1}\frac{t}{\left(
tb+(1-t)x\right) ^{2}}f^{\prime }\left( \frac{bx}{tb+(1-t)x}\right)
dt\right\}   \notag
\end{eqnarray}
\end{lemma}

Using Lemma \ref{1.5}, In \cite{I13e}, Iscan established the following
results which hold for harmonically $s$-convex functions.

\begin{theorem}
\label{1.6}Let $f:I\subset \left( 0,\infty \right) \rightarrow 
\mathbb{R}
$ be a differentiable function on $I^{\circ }$, $a,b\in I$ with $a<b,$ and $%
f^{\prime }\in L[a,b].$ If $\left\vert f^{\prime }\right\vert ^{q}$ is
harmonically $s$- convex on $[a,b]$ for $q\geq 1,$ then for all $x\in \left[
a,b\right] $, we have%
\begin{equation}
\left\vert f(x)-\frac{ab}{b-a}\dint\limits_{a}^{b}\frac{f(u)}{u^{2}}%
du\right\vert  \label{1-6}
\end{equation}%
\begin{eqnarray*}
&\leq &\frac{ab}{b-a}\left\{ \left( x-a\right) ^{2}\left( \lambda
_{1}(a,x,s,q,q)\left\vert f^{\prime }\left( x\right) \right\vert
^{q}+\lambda _{2}((a,x,s,q,q)\left\vert f^{\prime }\left( a\right)
\right\vert ^{q}\right) ^{\frac{1}{q}}\right. \\
&&+\left. \left( b-x\right) ^{2}\left( \lambda _{3}(b,x,s,q,q)\left\vert
f^{\prime }\left( x\right) \right\vert ^{q}+\lambda
_{4}(b,x,s,q,q)\left\vert f^{\prime }\left( b\right) \right\vert ^{q}\right)
^{\frac{1}{q}}\right\} ,
\end{eqnarray*}%
where 
\begin{equation*}
\lambda _{1}(a,x,s,\vartheta ,\rho )=\frac{\beta \left( \rho +s+1,1\right) }{%
x^{2\vartheta }}._{2}F_{1}\left( 2\vartheta ,\rho +s+1;\rho +s+2;1-\frac{a}{x%
}\right) ,
\end{equation*}%
\begin{equation*}
\lambda _{2}(a,x,s,\vartheta ,\rho )=\frac{\beta \left( \rho +1,1\right) }{%
x^{2\vartheta }}._{2}F_{1}\left( 2\vartheta ,\rho +1;\rho +s+2;1-\frac{a}{x}%
\right) ,
\end{equation*}%
\begin{equation*}
\lambda _{3}(b,x,s,\vartheta ,\rho )=\frac{\beta \left( 1,\rho +s+1\right) }{%
b^{2\vartheta }}._{2}F_{1}\left( 2\vartheta ,1;\rho +s+2;1-\frac{x}{b}%
\right) ,
\end{equation*}%
\begin{equation*}
\lambda _{4}(b,x,s,\vartheta ,\rho )=\frac{\beta \left( s+1,\rho +1\right) }{%
b^{2\vartheta }}._{2}F_{1}\left( 2\vartheta ,s+1;\rho +s+2;1-\frac{x}{b}%
\right) ,
\end{equation*}
\end{theorem}

\begin{theorem}
\label{1.7}Let $f:I\subset \left( 0,\infty \right) \rightarrow 
\mathbb{R}
$ be a differentiable function on $I^{\circ }$, $a,b\in I$ with $a<b,$ and $%
f^{\prime }\in L[a,b].$ If $\left\vert f^{\prime }\right\vert ^{q}$ is
harmonically $s$- convex on $[a,b]$ for $q\geq 1,$ then for all $x\in \left[
a,b\right] $, we have%
\begin{equation}
\left\vert f(x)-\frac{ab}{b-a}\dint\limits_{a}^{b}\frac{f(u)}{u^{2}}%
du\right\vert  \label{1-7}
\end{equation}%
\begin{eqnarray*}
&\leq &\frac{ab}{b-a}\left( \frac{1}{2}\right) ^{1-\frac{1}{q}}\left\{
\left( x-a\right) ^{2}\left( \lambda _{1}(a,x,s,q,1)\left\vert f^{\prime
}\left( x\right) \right\vert ^{q}+\lambda _{2}(a,x,s,q,1)\left\vert
f^{\prime }\left( a\right) \right\vert ^{q}\right) ^{\frac{1}{q}}\right. \\
&&+\left. \left( b-x\right) ^{2}\left( \lambda _{3}(b,x,s,q,1)\left\vert
f^{\prime }\left( x\right) \right\vert ^{q}+\lambda
_{4}(b,x,s,q,1)\left\vert f^{\prime }\left( b\right) \right\vert ^{q}\right)
^{\frac{1}{q}}\right\}
\end{eqnarray*}%
where $\lambda _{1}$, $\lambda _{2}$, $\lambda _{3}$ and $\lambda _{4}$ are
defined as in Theorem \ref{1.6}.
\end{theorem}

\begin{theorem}
\label{1.8}Let $f:I\subset \left( 0,\infty \right) \rightarrow 
\mathbb{R}
$ be a differentiable function on $I^{\circ }$, $a,b\in I$ with $a<b,$ and $%
f^{\prime }\in L[a,b].$ If $\left\vert f^{\prime }\right\vert ^{q}$ is
harmonically $s$- convex on $[a,b]$ for $q\geq 1,$ then for all $x\in \left[
a,b\right] $, we have%
\begin{equation}
\left\vert f(x)-\frac{ab}{b-a}\dint\limits_{a}^{b}\frac{f(u)}{u^{2}}%
du\right\vert  \label{1-8}
\end{equation}%
\begin{eqnarray*}
&\leq &\frac{ab}{b-a}\left\{ \lambda _{5}^{1-\frac{1}{q}}(a,x)\left(
x-a\right) ^{2}\left( \lambda _{1}(a,x,s,1,1)\left\vert f^{\prime }\left(
x\right) \right\vert ^{q}+\lambda _{2}(a,x,s,1,1)\left\vert f^{\prime
}\left( a\right) \right\vert ^{q}\right) ^{\frac{1}{q}}\right. \\
&&+\left. \lambda _{5}^{1-\frac{1}{q}}(b,x)\left( b-x\right) ^{2}\left(
\lambda _{3}(b,x,s,1,1)\left\vert f^{\prime }\left( x\right) \right\vert
^{q}+\lambda _{4}(b,x,s,1,1)\left\vert f^{\prime }\left( b\right)
\right\vert ^{q}\right) ^{\frac{1}{q}}\right\}
\end{eqnarray*}%
where 
\begin{equation*}
\lambda _{5}(\theta ,x)=\frac{1}{x-\theta }\left\{ \frac{1}{\theta }-\frac{%
\ln x-\ln \theta }{x-\theta }\right\} ,
\end{equation*}%
and $\lambda _{1}$, $\lambda _{2}$, $\lambda _{3}$ and $\lambda _{4}$ are
defined as in Theorem \ref{1.6}.
\end{theorem}

\begin{theorem}
\label{1.9}Let $f:I\subset \left( 0,\infty \right) \rightarrow 
\mathbb{R}
$ be a differentiable function on $I^{\circ }$, $a,b\in I$ with $a<b,$ and $%
f^{\prime }\in L[a,b].$ If $\left\vert f^{\prime }\right\vert ^{q}$ is
harmonically $s$-convex on $[a,b]$ for $q>1,\;\frac{1}{p}+\frac{1}{q}=1,$
then%
\begin{equation}
\left\vert f(x)-\frac{ab}{b-a}\dint\limits_{a}^{b}\frac{f(u)}{u^{2}}%
du\right\vert  \label{1-9}
\end{equation}%
\begin{eqnarray*}
&\leq &\frac{ab}{b-a}\left( \frac{1}{p+1}\right) ^{\frac{1}{p}}\left\{
\left( x-a\right) ^{2}\left( \lambda _{1}(a,x,s,q,0)\left\vert f^{\prime
}\left( x\right) \right\vert ^{q}+\lambda _{2}(a,x,s,q,0)\left\vert
f^{\prime }\left( a\right) \right\vert ^{q}\right) ^{\frac{1}{q}}\right. \\
&&+\left. \left( b-x\right) ^{2}\left( \lambda _{3}(b,x,s,q,0)\left\vert
f^{\prime }\left( x\right) \right\vert ^{q}+\lambda
_{4}(b,x,s,q,0)\left\vert f^{\prime }\left( b\right) \right\vert ^{q}\right)
^{\frac{1}{q}}\right\} .
\end{eqnarray*}%
where $\lambda _{1}$, $\lambda _{2}$, $\lambda _{3}$ and $\lambda _{4}$ are
defined as in Theorem \ref{1.6}.
\end{theorem}

\begin{theorem}
\label{1.10}Let $f:I\subset \left( 0,\infty \right) \rightarrow 
\mathbb{R}
$ be a differentiable function on $I^{\circ }$, $a,b\in I$ with $a<b,$ and $%
f^{\prime }\in L[a,b].$ If $\left\vert f^{\prime }\right\vert ^{q}$ is
harmonically $s$-convex on $[a,b]$ for $q>1,\;\frac{1}{p}+\frac{1}{q}=1,$
then%
\begin{equation}
\left\vert f(x)-\frac{ab}{b-a}\dint\limits_{a}^{b}\frac{f(u)}{u^{2}}%
du\right\vert  \label{1-10}
\end{equation}%
\begin{eqnarray*}
&\leq &\frac{ab}{b-a}\left\{ \left( \lambda _{1}(a,x,0,p,p)\right) ^{\frac{1%
}{p}}\left( x-a\right) ^{2}\left( \frac{\left\vert f^{\prime }\left(
x\right) \right\vert ^{q}+\left\vert f^{\prime }\left( a\right) \right\vert
^{q}}{s+1}\right) ^{\frac{1}{q}}\right. \\
&&+\left. \left( \lambda _{3}(b,x,0,p,p)\right) ^{\frac{1}{p}}\left(
b-x\right) ^{2}\left( \frac{\left\vert f^{\prime }\left( x\right)
\right\vert ^{q}+\left\vert f^{\prime }\left( b\right) \right\vert ^{q}}{s+1}%
\right) ^{\frac{1}{q}}\right\} .
\end{eqnarray*}%
where $\lambda _{1}$, $\lambda _{2}$, $\lambda _{3}$ and $\lambda _{4}$ are
defined as in Theorem \ref{1.6}.
\end{theorem}

We give some necessary definitions and mathematical preliminaries of
fractional calculus theory which are used throughout this paper.

\begin{definition}
Let $f\in L\left[ a,b\right] $. The Riemann-Liouville integrals $%
J_{a+}^{\alpha }f$ and $J_{b-}^{\alpha }f$ of oder $\alpha >0$ with $a\geq 0$
are defined by%
\begin{equation*}
J_{a+}^{\alpha }f(x)=\frac{1}{\Gamma (\alpha )}\dint\limits_{a}^{x}\left(
x-t\right) ^{\alpha -1}f(t)dt,\ x>a
\end{equation*}

and%
\begin{equation*}
J_{b-}^{\alpha }f(x)=\frac{1}{\Gamma (\alpha )}\dint\limits_{x}^{b}\left(
t-x\right) ^{\alpha -1}f(t)dt,\ x<b
\end{equation*}%
respectively, where $\Gamma (\alpha )$ is the Gamma function defined by $%
\Gamma (\alpha )=$ $\dint\limits_{0}^{\infty }e^{-t}t^{\alpha -1}dt$ and $%
J_{a^{+}}^{0}f(x)=J_{b^{-}}^{0}f(x)=f(x).$
\end{definition}

In the case of $\alpha =1$, the fractional integral reduces to the classical
integral. Because of the wide application of Hermite-Hadamard type
inequalities and fractional integrals, many researchers extend their studies
to Hermite-Hadamard type inequalities involving fractional integrals not
limited to integer integrals. Recently, more and more Hermite-Hadamard
inequalities involving fractional integrals have been obtained for different
classes of functions; see \cite{I13b,I13c,I13d,SSYB13,S12,WFZ12}.

In \cite{I13f}, Iscan gave Hermite--Hadamard's inequalities for Harmonically
convex functions in fractional integral forms as follows:

\begin{theorem}
Let $f:I\subseteq \left( 0,\infty \right) \rightarrow 
\mathbb{R}
$ be a function such that $f\in L[a,b]$, where $a,b\in I$ with $a<b$. If $f$
is a harmonically convex function on $[a,b]$, then the following
inequalities for fractional integrals hold:%
\begin{equation}
f\left( \frac{2ab}{a+b}\right) \leq \frac{\Gamma (\alpha +1)}{2}\left( \frac{%
ab}{b-a}\right) ^{\alpha }\left\{ J_{1/a-}^{\alpha }\left( f\circ g\right)
(1/b)+J_{1/b+}^{\alpha }\left( f\circ g\right) (1/a)\right\} \leq \frac{%
f(a)+f(b)}{2}  \label{2-0}
\end{equation}%
with $\alpha >0$, where $g(u)=1/u.$
\end{theorem}

In this paper, new Ostrowski\ type inequalities for harmonically $s$-convex
functions via Riemann--Liouville fractional integral. An interesting feature
of our results is that they provide new estimates on these types of
inequalities for fractional integrals and these results have some relation
with \cite{I13e} for $\alpha =1.$

\section{Main Results}

Let $f:I\subseteq \left( 0,\infty \right) \rightarrow 
\mathbb{R}
$ be a differentiable function on $I^{\circ }$, the interior of $I$,
throughout this section we will take%
\begin{eqnarray*}
&&S_{f}\left( g;\alpha ;x,a,b\right) \\
&=&\left[ \left( \frac{x-a}{ax}\right) ^{\alpha }+\left( \frac{b-x}{bx}%
\right) ^{\alpha }\right] f(x)-\Gamma (\alpha +1)\left\{ J_{1/x-}^{\alpha
}\left( f\circ g\right) (1/b)+J_{1/x+}^{\alpha }\left( f\circ g\right)
(1/a)\right\}
\end{eqnarray*}%
where $a,b\in I$ with $a<b$, $\alpha >0$, $x\in \left[ a,b\right] $, $%
g(u)=1/u$ and $\Gamma $ is Euler Gamma function.

\begin{lemma}
\label{2.1}Let $f:I\subset 
\mathbb{R}
\backslash \left\{ 0\right\} \rightarrow 
\mathbb{R}
$ be a differentiable function on $I^{\circ }$ and $a,b\in I$ with $a<b$. If 
$f^{\prime }\in L[a,b]$ then for all $x\in \left[ a,b\right] $ and $\alpha
>0 $ we have: 
\begin{eqnarray}
&&S_{f}\left( g;\alpha ;x,a,b\right)  \label{2-1} \\
&=&\frac{\left( x-a\right) ^{\alpha +1}}{\left( ax\right) ^{\alpha -1}}%
\dint\limits_{0}^{1}\frac{t^{\alpha }}{\left( ta+(1-t)x\right) ^{2}}%
f^{\prime }\left( \frac{ax}{ta+(1-t)x}\right) dt  \notag \\
&&-\frac{\left( b-x\right) ^{\alpha +1}}{\left( bx\right) ^{\alpha -1}}%
\dint\limits_{0}^{1}\frac{t^{\alpha }}{\left( tb+(1-t)x\right) ^{2}}%
f^{\prime }\left( \frac{bx}{tb+(1-t)x}\right) dt  \notag
\end{eqnarray}
\end{lemma}

\begin{proof}
By integration by parts, we can state%
\begin{eqnarray}
&&ax(x-a)\dint\limits_{0}^{1}\frac{t^{\alpha }}{\left( ta+(1-t)x\right) ^{2}}%
f^{\prime }\left( \frac{ax}{ta+(1-t)x}\right) dt  \notag \\
&=&\left. t^{\alpha }f\left( \frac{ax}{ta+(1-t)x}\right) \right\vert
_{0}^{1}-\alpha \dint\limits_{0}^{1}t^{\alpha -1}f\left( \frac{ax}{ta+(1-t)x}%
\right) dt  \notag \\
&=&f(x)-\alpha \left( \frac{ax}{x-a}\right) ^{\alpha
}\dint\limits_{1/x}^{1/a}\left( \frac{1}{a}-u\right) ^{\alpha -1}f(\frac{1}{u%
})du  \notag \\
&=&f(x)-\Gamma (\alpha +1)\left( \frac{ax}{x-a}\right) ^{\alpha
}J_{1/x+}^{\alpha }\left( f\circ g\right) (1/a)  \label{2-1a}
\end{eqnarray}%
and 
\begin{eqnarray}
&&-bx(b-x)\dint\limits_{0}^{1}\frac{t^{\alpha }}{\left( tb+(1-t)x\right) ^{2}%
}f^{\prime }\left( \frac{bx}{tb+(1-t)x}\right) dt  \notag \\
&=&\left. t^{\alpha }f\left( \frac{bx}{tb+(1-t)x}\right) \right\vert
_{0}^{1}-\alpha \dint\limits_{0}^{1}t^{\alpha -1}f\left( \frac{bx}{tb+(1-t)x}%
\right) dt  \notag \\
&=&f(x)-\alpha \left( \frac{bx}{b-x}\right) ^{\alpha
}\dint\limits_{1/b}^{1/x}\left( u-\frac{1}{b}\right) ^{\alpha -1}f(\frac{1}{u%
})du  \notag \\
&=&f(x)-\Gamma (\alpha +1)\left( \frac{bx}{b-x}\right) ^{\alpha
}J_{1/x-}^{\alpha }\left( f\circ g\right) (1/b)  \label{2-1b}
\end{eqnarray}%
Multiplying both sides of (\ref{2-1a}) and (\ref{2-1b}) by $\left( \frac{x-a%
}{ax}\right) ^{\alpha }$ and $\left( \frac{b-x}{bx}\right) ^{\alpha }$,
respectively and adding of results obtained, we get the desired results.
\end{proof}

\begin{remark}
\bigskip In Lemma \ref{2.1}, if we take $\alpha =1$, then the identity (\ref%
{2-1}) reduces the identity (\ref{1-5}) of Lemma \ref{1.5}.
\end{remark}

Using this lemma, we can obtain the following fractional integral
inequalities.

\begin{theorem}
\label{2.2}Let $f:I\subset \left( 0,\infty \right) \rightarrow 
\mathbb{R}
$ be a differentiable function on $I^{\circ }$, $a,b\in I$ with $a<b,$ and $%
f^{\prime }\in L[a,b].$ If $\left\vert f^{\prime }\right\vert ^{q}$ is
harmonically $s$- convex on $[a,b]$ for $q\geq 1,$ then for all $x\in \left[
a,b\right] $, we have%
\begin{equation}
\left\vert S_{f}\left( g;\alpha ;x,a,b\right) \right\vert  \label{2-2}
\end{equation}%
\begin{eqnarray*}
&\leq &\frac{\left( x-a\right) ^{\alpha +1}}{\left( ax\right) ^{\alpha -1}}%
\left\{ \left( \lambda _{1}(a,x,s,q,\alpha q)\left\vert f^{\prime }\left(
x\right) \right\vert ^{q}+\lambda _{2}((a,x,s,q,\alpha q)\left\vert
f^{\prime }\left( a\right) \right\vert ^{q}\right) ^{\frac{1}{q}}\right\} \\
&&+\frac{\left( b-x\right) ^{\alpha +1}}{\left( bx\right) ^{\alpha -1}}%
\left\{ \left( \lambda _{3}(b,x,s,q,\alpha q)\left\vert f^{\prime }\left(
x\right) \right\vert ^{q}+\lambda _{4}(b,x,s,q,\alpha q)\left\vert f^{\prime
}\left( b\right) \right\vert ^{q}\right) ^{\frac{1}{q}}\right\} ,
\end{eqnarray*}%
where 
\begin{equation*}
\lambda _{1}(a,x,s,\vartheta ,\rho )=\frac{\beta \left( \rho +s+1,1\right) }{%
x^{2\vartheta }}._{2}F_{1}\left( 2\vartheta ,\rho +s+1;\rho +s+2;1-\frac{a}{x%
}\right) ,
\end{equation*}%
\begin{equation*}
\lambda _{2}(a,x,s,\vartheta ,\rho )=\frac{\beta \left( \rho +1,s+1\right) }{%
x^{2\vartheta }}._{2}F_{1}\left( 2\vartheta ,\rho +1;\rho +s+2;1-\frac{a}{x}%
\right) ,
\end{equation*}%
\begin{equation*}
\lambda _{3}(b,x,s,\vartheta ,\rho )=\frac{\beta \left( 1,\rho +s+1\right) }{%
b^{2\vartheta }}._{2}F_{1}\left( 2\vartheta ,1;\rho +s+2;1-\frac{x}{b}%
\right) ,
\end{equation*}%
\begin{equation*}
\lambda _{4}(b,x,s,\vartheta ,\rho )=\frac{\beta \left( s+1,\rho +1\right) }{%
b^{2\vartheta }}._{2}F_{1}\left( 2\vartheta ,s+1;\rho +s+2;1-\frac{x}{b}%
\right) .
\end{equation*}
\end{theorem}

\begin{proof}
From Lemma \ref{2.1} and power mean inequality and using the harmonically $s$%
-convexity of $\left\vert f^{\prime }\right\vert ^{q}$ on $[a,b],$we have%
\begin{eqnarray*}
&&\left\vert S_{f}\left( g;\alpha ;x,a,b\right) \right\vert \\
&\leq &\frac{\left( x-a\right) ^{\alpha +1}}{\left( ax\right) ^{\alpha -1}}%
\dint\limits_{0}^{1}\frac{t^{\alpha }}{\left( ta+(1-t)x\right) ^{2}}%
\left\vert f^{\prime }\left( \frac{ax}{ta+(1-t)x}\right) \right\vert dt \\
&&+\frac{\left( b-x\right) ^{\alpha +1}}{\left( bx\right) ^{\alpha -1}}%
\dint\limits_{0}^{1}\frac{t^{\alpha }}{\left( tb+(1-t)x\right) ^{2}}%
\left\vert f^{\prime }\left( \frac{bx}{tb+(1-t)x}\right) \right\vert dt
\end{eqnarray*}%
\begin{eqnarray}
&\leq &\frac{\left( x-a\right) ^{\alpha +1}}{\left( ax\right) ^{\alpha -1}}%
\left( \dint\limits_{0}^{1}1dt\right) ^{1-\frac{1}{q}}  \label{2-2a} \\
&&\times \left( \dint\limits_{0}^{1}\frac{t^{\alpha q}}{\left(
ta+(1-t)x\right) ^{2q}}\left[ t^{s}\left\vert f^{\prime }\left( x\right)
\right\vert ^{q}+(1-t)^{s}\left\vert f^{\prime }\left( a\right) \right\vert
^{q}\right] dt\right) ^{\frac{1}{q}}  \notag
\end{eqnarray}%
\begin{eqnarray*}
&&+\frac{\left( b-x\right) ^{\alpha +1}}{\left( bx\right) ^{\alpha -1}}%
\left( \dint\limits_{0}^{1}1dt\right) ^{1-\frac{1}{q}} \\
&&\times \left( \dint\limits_{0}^{1}\frac{t^{\alpha q}}{\left(
tb+(1-t)x\right) ^{2q}}\left[ t^{s}\left\vert f^{\prime }\left( x\right)
\right\vert ^{q}+(1-t)^{s}\left\vert f^{\prime }\left( b\right) \right\vert
^{q}\right] dt\right) ^{\frac{1}{q}},
\end{eqnarray*}%
where an easy calculation gives%
\begin{equation}
\dint\limits_{0}^{1}\frac{t^{\alpha q+s}}{\left( ta+(1-t)x\right) ^{2q}}dt=%
\frac{\beta \left( \alpha q+s+1,1\right) }{x^{2q}}._{2}F_{1}\left( 2q,\alpha
q+s+1;\alpha q+s+2;1-\frac{a}{x}\right) ,  \label{2-2b}
\end{equation}%
\begin{equation*}
\dint\limits_{0}^{1}\frac{t^{\alpha q}(1-t)^{s}}{\left( ta+(1-t)x\right)
^{2q}}dt=\frac{\beta \left( \alpha q+1,s+1\right) }{x^{2q}}._{2}F_{1}\left(
2q,\alpha q+1;s+\alpha q+2;1-\frac{a}{x}\right) ,
\end{equation*}%
\begin{equation*}
\dint\limits_{0}^{1}\frac{t^{\alpha q+s}}{\left( tb+(1-t)x\right) ^{2q}}dt=%
\frac{\beta \left( 1,\alpha q+s+1\right) }{b^{2q}}._{2}F_{1}\left(
2q,1;\alpha q+s+2;1-\frac{x}{b}\right) ,
\end{equation*}%
\begin{equation}
\dint\limits_{0}^{1}\frac{t^{\alpha q}(1-t)^{s}}{\left( tb+(1-t)x\right)
^{2q}}dt=\frac{\beta \left( s+1,\alpha q+1\right) }{b^{2q}}._{2}F_{1}\left(
2q,s+1;s+\alpha q+2;1-\frac{x}{b}\right) .  \label{2-2c}
\end{equation}%
Hence, If we use (\ref{2-2b})-(\ref{2-2c}) in (\ref{2-2a}), we obtain the
desired result. This completes the proof.
\end{proof}

\begin{remark}
\bigskip In Theorem \ref{2.2}, if we take $\alpha =1$, then the identity (%
\ref{2-2}) reduces the identity (\ref{1-6}) of Theorem \ref{2.6}.
\end{remark}

\begin{corollary}
\bigskip In Theorem \ref{2.2}, if $|f^{\prime }(x)|\leq M$, $x\in \left[ a,b%
\right] ,$ then inequality%
\begin{equation*}
\left\vert S_{f}\left( g;\alpha ;x,a,b\right) \right\vert
\end{equation*}%
\begin{eqnarray*}
&\leq &M\left\{ \frac{\left( x-a\right) ^{\alpha +1}}{\left( ax\right)
^{\alpha -1}}\left[ \lambda _{1}(a,x,s,q,\alpha q)+\lambda
_{2}((a,x,s,q,\alpha q)\right] ^{\frac{1}{q}}\right. \\
&&\left. +\frac{\left( b-x\right) ^{\alpha +1}}{\left( bx\right) ^{\alpha -1}%
}\left[ \lambda _{3}(b,x,s,q,\alpha q)+\lambda _{4}(b,x,s,q,\alpha q)\right]
^{\frac{1}{q}}\right\} ,
\end{eqnarray*}%
holds.
\end{corollary}

\begin{theorem}
\label{2.3}Let $f:I\subset \left( 0,\infty \right) \rightarrow 
\mathbb{R}
$ be a differentiable function on $I^{\circ }$, $a,b\in I$ with $a<b,$ and $%
f^{\prime }\in L[a,b].$ If $\left\vert f^{\prime }\right\vert ^{q}$ is
harmonically $s$- convex on $[a,b]$ for $q\geq 1,$ then for all $x\in \left[
a,b\right] $, we have%
\begin{equation}
\left\vert S_{f}\left( g;\alpha ;x,a,b\right) \right\vert  \label{2-3}
\end{equation}%
\begin{eqnarray*}
&\leq &\left( \frac{1}{\alpha +1}\right) ^{1-\frac{1}{q}}\left\{ \frac{%
\left( x-a\right) ^{\alpha +1}}{\left( ax\right) ^{\alpha -1}}\left( \lambda
_{1}(a,x,s,q,\alpha )\left\vert f^{\prime }\left( x\right) \right\vert
^{q}+\lambda _{2}(a,x,s,q,\alpha )\left\vert f^{\prime }\left( a\right)
\right\vert ^{q}\right) ^{\frac{1}{q}}\right. \\
&&+\left. \frac{\left( b-x\right) ^{\alpha +1}}{\left( bx\right) ^{\alpha -1}%
}\left( \lambda _{3}(b,x,s,q,\alpha )\left\vert f^{\prime }\left( x\right)
\right\vert ^{q}+\lambda _{4}(b,x,s,q,\alpha )\left\vert f^{\prime }\left(
b\right) \right\vert ^{q}\right) ^{\frac{1}{q}}\right\}
\end{eqnarray*}%
where $\alpha >0$ and $\lambda _{1}$, $\lambda _{2}$, $\lambda _{3}$ and $%
\lambda _{4}$ are defined as in Theorem \ref{2.2}.
\end{theorem}

\begin{proof}
From Lemma \ref{2.1}, power mean inequality and the harmonically $s$%
-convexity of $\left\vert f^{\prime }\right\vert ^{q}$ on $[a,b],$we have%
\begin{eqnarray*}
&&\left\vert S_{f}\left( g;\alpha ;x,a,b\right) \right\vert \\
&\leq &\frac{\left( x-a\right) ^{\alpha +1}}{\left( ax\right) ^{\alpha -1}}%
\left( \dint\limits_{0}^{1}t^{\alpha }dt\right) ^{1-\frac{1}{q}} \\
&&\times \left( \dint\limits_{0}^{1}\frac{t^{\alpha }}{\left(
ta+(1-t)x\right) ^{2q}}\left[ t^{s}\left\vert f^{\prime }\left( x\right)
\right\vert ^{q}+(1-t)^{s}\left\vert f^{\prime }\left( a\right) \right\vert
^{q}\right] dt\right) ^{\frac{1}{q}}
\end{eqnarray*}%
\begin{eqnarray*}
&&+\frac{\left( b-x\right) ^{\alpha +1}}{\left( bx\right) ^{\alpha -1}}%
\left( \dint\limits_{0}^{1}t^{\alpha }dt\right) ^{1-\frac{1}{q}} \\
&&\times \left( \dint\limits_{0}^{1}\frac{t^{\alpha }}{\left(
tb+(1-t)x\right) ^{2q}}\left[ t^{s}\left\vert f^{\prime }\left( x\right)
\right\vert ^{q}+(1-t)^{s}\left\vert f^{\prime }\left( b\right) \right\vert
^{q}\right] dt\right) ^{\frac{1}{q}}
\end{eqnarray*}%
\begin{eqnarray*}
&\leq &\left( \frac{1}{\alpha +1}\right) ^{1-\frac{1}{q}}\left\{ \frac{%
\left( x-a\right) ^{\alpha +1}}{\left( ax\right) ^{\alpha -1}}\left( \lambda
_{1}(a,x,s,q,\alpha )\left\vert f^{\prime }\left( x\right) \right\vert
^{q}+\lambda _{2}(a,x,s,q,\alpha )\left\vert f^{\prime }\left( a\right)
\right\vert ^{q}\right) ^{\frac{1}{q}}\right. \\
&&+\left. \frac{\left( b-x\right) ^{\alpha +1}}{\left( bx\right) ^{\alpha -1}%
}\left( \lambda _{3}(b,x,s,q,\alpha )\left\vert f^{\prime }\left( x\right)
\right\vert ^{q}+\lambda _{4}(b,x,s,q,\alpha )\left\vert f^{\prime }\left(
b\right) \right\vert ^{q}\right) ^{\frac{1}{q}}\right\}
\end{eqnarray*}%
This completes the proof.
\end{proof}

\begin{remark}
\bigskip In Theorem \ref{2.3}, if we take $\alpha =1$, then the identity (%
\ref{2-3}) reduces the identity (\ref{1-7}) of Theorem \ref{1.7}.
\end{remark}

\begin{corollary}
\bigskip In Theorem \ref{2.3}, if $|f^{\prime }(x)|\leq M$, $x\in \left[ a,b%
\right] ,$ then inequality%
\begin{equation*}
\left\vert S_{f}\left( g;\alpha ;x,a,b\right) \right\vert \leq M\left( \frac{%
1}{\alpha +1}\right) ^{1-\frac{1}{q}}
\end{equation*}%
\begin{eqnarray*}
&&\times \left\{ \frac{\left( x-a\right) ^{\alpha +1}}{\left( ax\right)
^{\alpha -1}}\left[ \lambda _{1}(a,x,s,q,\alpha )+\lambda
_{2}(a,x,s,q,\alpha )\right] ^{\frac{1}{q}}\right. \\
&&\left. +\frac{\left( b-x\right) ^{\alpha +1}}{\left( bx\right) ^{\alpha -1}%
}\left[ \lambda _{3}(b,x,s,q,\alpha )+\lambda _{4}(b,x,s,q,\alpha ))\right]
^{\frac{1}{q}}\right\} ,
\end{eqnarray*}%
holds.
\end{corollary}

\begin{theorem}
\label{2.4}Let $f:I\subset \left( 0,\infty \right) \rightarrow 
\mathbb{R}
$ be a differentiable function on $I^{\circ }$, $a,b\in I$ with $a<b,$ and $%
f^{\prime }\in L[a,b].$ If $\left\vert f^{\prime }\right\vert ^{q}$ is
harmonically $s$- convex on $[a,b]$ for $q\geq 1,$ then for all $x\in \left[
a,b\right] $, we have%
\begin{equation}
\left\vert S_{f}\left( g;\alpha ;x,a,b\right) \right\vert  \label{2-4}
\end{equation}%
\begin{eqnarray*}
&\leq &\left( \frac{1}{\alpha +1}\right) ^{1-\frac{1}{q}}\left\{ \lambda
_{5}^{1-\frac{1}{q}}(a,x,\alpha )\frac{\left( x-a\right) ^{\alpha +1}}{%
\left( ax\right) ^{\alpha -1}}\left( \lambda _{1}(a,x,s,1,\alpha )\left\vert
f^{\prime }\left( x\right) \right\vert ^{q}+\lambda _{2}(a,x,s,1,\alpha
)\left\vert f^{\prime }\left( a\right) \right\vert ^{q}\right) ^{\frac{1}{q}%
}\right. \\
&&+\left. \lambda _{6}^{1-\frac{1}{q}}(b,x,\alpha )\frac{\left( b-x\right)
^{\alpha +1}}{\left( bx\right) ^{\alpha -1}}\left( \lambda
_{3}(b,x,s,1,\alpha )\left\vert f^{\prime }\left( x\right) \right\vert
^{q}+\lambda _{4}(b,x,s,1,\alpha )\left\vert f^{\prime }\left( b\right)
\right\vert ^{q}\right) ^{\frac{1}{q}}\right\}
\end{eqnarray*}%
where 
\begin{eqnarray*}
\lambda _{5}(a,x,\alpha ) &=&\frac{1}{x^{2}}._{2}F_{1}(2,\alpha +1;\alpha
+2;1-\frac{a}{x}), \\
\lambda _{6}(a,x,\alpha ) &=&\frac{1}{b^{2}}._{2}F_{1}(2,1;\alpha +2;1-\frac{%
x}{b})
\end{eqnarray*}%
and $\lambda _{1}$, $\lambda _{2}$, $\lambda _{3}$ and $\lambda _{4}$ are
defined as in Theorem \ref{2.2}.
\end{theorem}

\begin{proof}
From Lemma \ref{2.1}, Power mean inequality and the harmonically $s$%
-convexity of $\left\vert f^{\prime }\right\vert ^{q}$ on $[a,b],$we have%
\begin{eqnarray}
&&\left\vert S_{f}\left( g;\alpha ;x,a,b\right) \right\vert  \label{2-4a} \\
&\leq &\frac{\left( x-a\right) ^{\alpha +1}}{\left( ax\right) ^{\alpha -1}}%
\left( \dint\limits_{0}^{1}\frac{t^{\alpha }}{\left( ta+(1-t)x\right) ^{2}}%
dt\right) ^{1-\frac{1}{q}}  \notag \\
&&\times \left( \dint\limits_{0}^{1}\frac{t^{\alpha }}{\left(
ta+(1-t)x\right) ^{2}}\left[ t^{s}\left\vert f^{\prime }\left( x\right)
\right\vert ^{q}+(1-t)^{s}\left\vert f^{\prime }\left( a\right) \right\vert
^{q}\right] dt\right) ^{\frac{1}{q}}  \notag
\end{eqnarray}%
\begin{eqnarray*}
&&+\frac{\left( b-x\right) ^{\alpha +1}}{\left( bx\right) ^{\alpha -1}}%
\left( \dint\limits_{0}^{1}\frac{t^{\alpha }}{\left( tb+(1-t)x\right) ^{2}}%
dt\right) ^{1-\frac{1}{q}} \\
&&\times \left( \dint\limits_{0}^{1}\frac{t^{\alpha }}{\left(
tb+(1-t)x\right) ^{2}}\left[ t^{s}\left\vert f^{\prime }\left( x\right)
\right\vert ^{q}+(1-t)^{s}\left\vert f^{\prime }\left( b\right) \right\vert
^{q}\right] dt\right) ^{\frac{1}{q}}.
\end{eqnarray*}%
It is easily check that%
\begin{equation}
\dint\limits_{0}^{1}\frac{t^{\alpha }}{\left( ta+(1-t)x\right) ^{2}}dt=\frac{%
1}{x^{2}(\alpha +1)}._{2}F_{1}(2,\alpha +1;\alpha +2;1-\frac{a}{x}),
\label{2-4b}
\end{equation}%
\begin{equation*}
\dint\limits_{0}^{1}\frac{t^{\alpha }}{\left( tb+(1-t)x\right) ^{2}}dt=\frac{%
1}{b^{2}(\alpha +1)},_{2}F_{1}(2,1;\alpha +2;1-\frac{x}{b}).
\end{equation*}%
Hence, If we use (\ref{2-2b})-(\ref{2-2c}) \ for $q=1$ and (\ref{2-4b}) in (%
\ref{2-4a}), we obtain the desired result. This completes the proof.
\end{proof}

\begin{remark}
\bigskip In Theorem \ref{2.4}, if we take $\alpha =1$, then the identity (%
\ref{2-4}) reduces the identity (\ref{1-8}) of Theorem \ref{1.8}.
\end{remark}

\begin{corollary}
\bigskip In Theorem \ref{2.4}, if $|f^{\prime }(x)|\leq M$, $x\in \left[ a,b%
\right] ,$ then inequality%
\begin{equation*}
\left\vert S_{f}\left( g;\alpha ;x,a,b\right) \right\vert \leq M\left( \frac{%
1}{\alpha +1}\right) ^{1-\frac{1}{q}}
\end{equation*}%
\begin{eqnarray*}
&&\times \left\{ \lambda _{5}^{1-\frac{1}{q}}(a,x,\alpha )\frac{\left(
x-a\right) ^{\alpha +1}}{\left( ax\right) ^{\alpha -1}}\left[ \lambda
_{1}(a,x,s,1,\alpha )+\lambda _{2}(a,x,s,1,\alpha )\right] ^{\frac{1}{q}%
}\right. \\
&&+\left. \lambda _{6}^{1-\frac{1}{q}}(b,x,\alpha )\frac{\left( b-x\right)
^{\alpha +1}}{\left( bx\right) ^{\alpha -1}}\left[ \lambda
_{3}(b,x,s,1,\alpha )+\lambda _{4}(b,x,s,1,\alpha )\right] ^{\frac{1}{q}%
}\right\}
\end{eqnarray*}%
holds.
\end{corollary}

\begin{theorem}
\label{2.5}Let $f:I\subset \left( 0,\infty \right) \rightarrow 
\mathbb{R}
$ be a differentiable function on $I^{\circ }$, $a,b\in I$ with $a<b,$ and $%
f^{\prime }\in L[a,b].$ If $\left\vert f^{\prime }\right\vert ^{q}$ is
harmonically $s$-convex on $[a,b]$ for $q>1,\;\frac{1}{p}+\frac{1}{q}=1,$
then%
\begin{equation}
\left\vert S_{f}\left( g;\alpha ;x,a,b\right) \right\vert  \label{2-5}
\end{equation}%
\begin{eqnarray*}
&\leq &\left( \frac{1}{\alpha p+1}\right) ^{\frac{1}{p}}\left\{ \frac{\left(
x-a\right) ^{\alpha +1}}{\left( ax\right) ^{\alpha -1}}\left( \lambda
_{1}(a,x,s,q,0)\left\vert f^{\prime }\left( x\right) \right\vert
^{q}+\lambda _{2}(a,x,s,q,0)\left\vert f^{\prime }\left( a\right)
\right\vert ^{q}\right) ^{\frac{1}{q}}\right. \\
&&+\left. \frac{\left( b-x\right) ^{\alpha +1}}{\left( bx\right) ^{\alpha -1}%
}\left( \lambda _{3}(b,x,s,q,0)\left\vert f^{\prime }\left( x\right)
\right\vert ^{q}+\lambda _{4}(b,x,s,q,0)\left\vert f^{\prime }\left(
b\right) \right\vert ^{q}\right) ^{\frac{1}{q}}\right\} .
\end{eqnarray*}%
where $\lambda _{1}$, $\lambda _{2}$, $\lambda _{3}$ and $\lambda _{4}$ are
defined as in Theorem \ref{2.2}.
\end{theorem}

\begin{proof}
From Lemma \ref{2.1}, H\"{o}lder's inequality and the harmonically convexity
of $\left\vert f^{\prime }\right\vert ^{q}$ on $[a,b],$we have%
\begin{eqnarray*}
&&\left\vert S_{f}\left( g;\alpha ;x,a,b\right) \right\vert \\
&\leq &\frac{\left( x-a\right) ^{\alpha +1}}{\left( ax\right) ^{\alpha -1}}%
\left( \dint\limits_{0}^{1}t^{\alpha p}dt\right) ^{\frac{1}{p}} \\
&&\times \left( \dint\limits_{0}^{1}\frac{1}{\left( ta+(1-t)x\right) ^{2q}}%
\left[ t^{s}\left\vert f^{\prime }\left( x\right) \right\vert
^{q}+(1-t)^{s}\left\vert f^{\prime }\left( a\right) \right\vert ^{q}\right]
dt\right) ^{\frac{1}{q}}
\end{eqnarray*}%
\begin{eqnarray*}
&&+\frac{\left( b-x\right) ^{\alpha +1}}{\left( bx\right) ^{\alpha -1}}%
\left( \dint\limits_{0}^{1}t^{\alpha p}dt\right) ^{\frac{1}{p}} \\
&&\times \left( \dint\limits_{0}^{1}\frac{1}{\left( tb+(1-t)x\right) ^{2q}}%
\left[ t^{s}\left\vert f^{\prime }\left( x\right) \right\vert
^{q}+(1-t)^{s}\left\vert f^{\prime }\left( b\right) \right\vert ^{q}\right]
dt\right) ^{\frac{1}{q}}
\end{eqnarray*}%
\begin{eqnarray*}
&\leq &\left( \frac{1}{\alpha p+1}\right) ^{\frac{1}{p}}\left\{ \frac{\left(
x-a\right) ^{\alpha +1}}{\left( ax\right) ^{\alpha -1}}\left( \lambda
_{1}(a,x,s,q,0)\left\vert f^{\prime }\left( x\right) \right\vert
^{q}+\lambda _{2}(a,x,s,q,0)\left\vert f^{\prime }\left( a\right)
\right\vert ^{q}\right) ^{\frac{1}{q}}\right. \\
&&+\left. \frac{\left( b-x\right) ^{\alpha +1}}{\left( bx\right) ^{\alpha -1}%
}\left( \lambda _{3}(b,x,s,q,0)\left\vert f^{\prime }\left( x\right)
\right\vert ^{q}+\lambda _{4}(b,x,s,q,0)\left\vert f^{\prime }\left(
b\right) \right\vert ^{q}\right) ^{\frac{1}{q}}\right\} .
\end{eqnarray*}%
This completes the proof.
\end{proof}

\begin{remark}
\bigskip In Theorem \ref{2.5}, if we take $\alpha =1$, then the identity (%
\ref{2-5}) reduces the identity (\ref{1-9}) of Theorem \ref{1.9}.
\end{remark}

\begin{corollary}
\bigskip In Theorem \ref{2.5}, if $|f^{\prime }(x)|\leq M$, $x\in \left[ a,b%
\right] ,$ then inequality%
\begin{equation*}
\left\vert S_{f}\left( g;\alpha ;x,a,b\right) \right\vert \leq M\left( \frac{%
1}{\alpha p+1}\right) ^{\frac{1}{p}}
\end{equation*}%
\begin{eqnarray*}
&&\times \left\{ \frac{\left( x-a\right) ^{\alpha +1}}{\left( ax\right)
^{\alpha -1}}\left( \lambda _{1}(a,x,s,q,0)+\lambda _{2}(a,x,s,q,0)\right) ^{%
\frac{1}{q}}\right. \\
&&+\left. \frac{\left( b-x\right) ^{\alpha +1}}{\left( bx\right) ^{\alpha -1}%
}\left( \lambda _{3}(b,x,s,q,0)+\lambda _{4}(b,x,s,q,0)\right) ^{\frac{1}{q}%
}\right\} .
\end{eqnarray*}%
holds.
\end{corollary}

\begin{theorem}
\label{2.6}Let $f:I\subset \left( 0,\infty \right) \rightarrow 
\mathbb{R}
$ be a differentiable function on $I^{\circ }$, $a,b\in I$ with $a<b,$ and $%
f^{\prime }\in L[a,b].$ If $\left\vert f^{\prime }\right\vert ^{q}$ is
harmonically $s$-convex on $[a,b]$ for $q>1,\;\frac{1}{p}+\frac{1}{q}=1,$
then%
\begin{equation}
\left\vert S_{f}\left( g;\alpha ;x,a,b\right) \right\vert  \label{2-6}
\end{equation}%
\begin{eqnarray*}
&\leq &\left\{ \frac{\left( x-a\right) ^{\alpha +1}}{\left( ax\right)
^{\alpha -1}}\left( \lambda _{1}(a,x,0,p,\alpha p)\right) ^{\frac{1}{p}%
}\left( \frac{\left\vert f^{\prime }\left( x\right) \right\vert
^{q}+\left\vert f^{\prime }\left( a\right) \right\vert ^{q}}{s+1}\right) ^{%
\frac{1}{q}}\right. \\
&&+\left. \frac{\left( b-x\right) ^{\alpha +1}}{\left( bx\right) ^{\alpha -1}%
}\left( \lambda _{3}(b,x,0,p,\alpha p)\right) ^{\frac{1}{p}}\left( \frac{%
\left\vert f^{\prime }\left( x\right) \right\vert ^{q}+\left\vert f^{\prime
}\left( b\right) \right\vert ^{q}}{s+1}\right) ^{\frac{1}{q}}\right\} .
\end{eqnarray*}%
where $\lambda _{1}$, $\lambda _{2}$, $\lambda _{3}$ and $\lambda _{4}$ are
defined as in Theorem \ref{2.2}.
\end{theorem}

\begin{proof}
From Lemma \ref{2.1}, H\"{o}lder's inequality and the harmonically convexity
of $\left\vert f^{\prime }\right\vert ^{q}$ on $[a,b],$we have%
\begin{eqnarray*}
&&\left\vert S_{f}\left( g;\alpha ;x,a,b\right) \right\vert \\
&\leq &\frac{\left( x-a\right) ^{\alpha +1}}{\left( ax\right) ^{\alpha -1}}%
\left( \dint\limits_{0}^{1}\frac{t^{\alpha p}}{\left( ta+(1-t)x\right) ^{2p}}%
dt\right) ^{\frac{1}{p}} \\
&&\times \left( \dint\limits_{0}^{1}\left[ t^{s}\left\vert f^{\prime }\left(
x\right) \right\vert ^{q}+(1-t)^{s}\left\vert f^{\prime }\left( a\right)
\right\vert ^{q}\right] dt\right) ^{\frac{1}{q}}
\end{eqnarray*}%
\begin{eqnarray*}
&&+\frac{\left( b-x\right) ^{\alpha +1}}{\left( bx\right) ^{\alpha -1}}%
\left( \dint\limits_{0}^{1}\frac{t^{\alpha p}}{\left( tb+(1-t)x\right) ^{2p}}%
dt\right) ^{\frac{1}{p}} \\
&&\times \left( \dint\limits_{0}^{1}\left[ t^{s}\left\vert f^{\prime }\left(
x\right) \right\vert ^{q}+(1-t)^{s}\left\vert f^{\prime }\left( b\right)
\right\vert ^{q}\right] dt\right) ^{\frac{1}{q}}
\end{eqnarray*}%
\begin{eqnarray*}
&\leq &\left\{ \frac{\left( x-a\right) ^{\alpha +1}}{\left( ax\right)
^{\alpha -1}}\left( \lambda _{1}(a,x,0,p,\alpha p)\right) ^{\frac{1}{p}%
}\left( \frac{\left\vert f^{\prime }\left( x\right) \right\vert
^{q}+\left\vert f^{\prime }\left( a\right) \right\vert ^{q}}{s+1}\right) ^{%
\frac{1}{q}}\right. \\
&&+\left. \frac{\left( b-x\right) ^{\alpha +1}}{\left( bx\right) ^{\alpha -1}%
}\left( \lambda _{3}(b,x,0,p,\alpha p)\right) ^{\frac{1}{p}}\left( \frac{%
\left\vert f^{\prime }\left( x\right) \right\vert ^{q}+\left\vert f^{\prime
}\left( b\right) \right\vert ^{q}}{s+1}\right) ^{\frac{1}{q}}\right\} .
\end{eqnarray*}%
This completes the proof.
\end{proof}

\begin{remark}
\bigskip In Theorem \ref{2.6}, if we take $\alpha =1$, then the identity (%
\ref{2-6}) reduces the identity (\ref{1-10}) of Theorem \ref{1-10}.
\end{remark}

\begin{corollary}
\bigskip In Theorem \ref{2.5}, if $|f^{\prime }(x)|\leq M$, $x\in \left[ a,b%
\right] ,$ then inequality%
\begin{equation*}
\left\vert S_{f}\left( g;\alpha ;x,a,b\right) \right\vert \leq M\left( \frac{%
2}{s+1}\right) ^{\frac{1}{q}}
\end{equation*}%
\begin{equation*}
\times \left\{ \frac{\left( x-a\right) ^{\alpha +1}}{\left( ax\right)
^{\alpha -1}}\lambda _{1}^{\frac{1}{p}}(a,x,0,p,\alpha p)+\frac{\left(
b-x\right) ^{\alpha +1}}{\left( bx\right) ^{\alpha -1}}\lambda _{3}^{\frac{1%
}{p}}(b,x,0,p,\alpha p)\right\}
\end{equation*}%
holds.
\end{corollary}

\end{document}